 \documentclass[10 pt]{amsart}

\usepackage{amsmath}
\usepackage{amssymb}
\usepackage{amsfonts}
\usepackage{amsthm}
\usepackage{mathrsfs}
\usepackage{fancyhdr}
\usepackage[all]{xy}
\usepackage{setspace}
\usepackage{hyperref}
\usepackage{paralist}
\usepackage{multirow}
\usepackage{array}
\usepackage{tikz}
\usepackage{float}
\usepackage[linesnumbered,ruled,vlined]{algorithm2e}

\newtheorem{theorem}{Theorem}[section]

\theoremstyle{proposition}

\theoremstyle{corollary}

\theoremstyle{definition}

\theoremstyle{remark}

\numberwithin{equation}{section}

\begin{document}
\author{Kubra Nari}
\email{nari15@itu.edu.tr}
\address{Informatics Institute, Istanbul Technical University, Turkey}
\author{Enver Ozdemir}
\email{ozdemiren@itu.edu.tr}
\address{Informatics Institute, Istanbul Technical University, Turkey}

\author{Neslihan Aysen Ozkirisci}
\email{aozk@yildiz.edu.tr}
\address{Department of Mathematics,Yildiz Technical University, Turkey}

\title{Strong Pseudo Primes To Base 2}
\begin{abstract}
In this work, we add an additional condition to strong pseudo prime test to base 2. Then, we provide theoretical and heuristics evidences  showing that the resulting algorithm catches all composite numbers. Our method is based on the structure of singular cubics' Jacobian groups on which we also define an effective addition algorithm. 
\end{abstract}

\maketitle

\specialsection*{\bf \Large Introduction}

\indent It has been an interest of scientists that for a given integer $n$ construct a method to decide if $n$ has a non-trivial divisor or not. In theory, it is very easy to say whether an integer has a non-trivial factor or not since  a number is composite if and only if it has a factor less than $\sqrt{n}$. The popularity of digital communication along with cryptographic methods which employ prime numbers forces to design an efficient way in practice deciding if $n$ has a factor or not. Even though, the first algorithms for primality utilize basic group theory, they are still more practical and popular than the sophisticated ones which utilize advanced mathematical tools like elliptic curve groups and cyclotomic fields. \\
\indent Let $n$ be an integer, the multiplicative group modulo $n$ has order $n-1$ if $n$ has no non-trivial factor. This fact has been known for a long time and employed in many ways. For example an easy consequence of it is Fermat's little theorem. Another application \cite{RAB} was presented by  M. Rabin who offered a method based on this to detect composite integers. His method works as follows:
\begin{enumerate}
	\item For a given odd integer $n$, find $e$ and odd integer $m$ such that $$ n-1=2^em$$
	\item Select a random number $a$.
	\item See if  $a^{2^im}=c_i$ is congruent to 1 or -1 modulo $n$ for any $0 \le i \le e-1$.  
	\item If none of $c_i$ is 1 or -1 for $i=0,\dots, e-1$, then declare $n$ a composite number.
\end{enumerate}
   
 Unfortunately, the test doesn't say primality of the number $a$ if one of $c_i$ is 1 or -1. In general, this test repeats for a few distinct $a$'s and  the number $a$ is selected to be a prime  starting with $2$ and continue with consecutive primes 3, 5, 7, 11 and so on. The  test with a random integer $a$  is called strong pseudo prime test to base $a$. If the test doesn't say the number $n$ is composite then $n$ is called a {\it strong pseudo prime to base $a$.} The set of all strong pseudo primes to base $a$ is denoted by $spsp(a$).  If an integer $n\in spsp(a) $ then the probability of $n$ being a prime is more than $75\%$ \cite{SCHOOF}. G. Miller \cite{MIL} later showed that if $n$ is $spsp(a)$ for all prime integers $a\le 2\log^2 n$ then $n$ must be prime if GRH is correct. \\
 \indent Two significant algorithms \cite{ADLPOM, ATKMOR} were presented for  primality testing in 80s. The first one describes a deterministic method to decide whether an integer $n$ is really a prime number, but it fails to be a polynomial time algorithm \cite{ADLPOM, SCHOOF}. The second one decides an integer $n$ is   prime if it can construct an elliptic curve group $E$ modulo $n$ such that $E$ has a point of  order larger than a certain integer depending on $n$. This algorithm is much faster than the first one but it might not be able to produce the desired elliptic curve group $E$ therefore the algorithm is probabilistic \cite{SCHOOF}. We should also mention that in early 2000, a deterministic and polynomial time algorithm \cite{AGR} was presented but for a given integer $n$ the algorithm's running time is $O(\log^{12+\epsilon}n)$ for $\epsilon >0$ \cite{SCHOOF}.\\
 \indent The aforementioned \cite{ADLPOM,ATKMOR} algorithms for primality test are not as efficient as the strong pseudo prime (spsp) test. Even though, the spsp test doesn't guarantee the primality of an integer, it has been widely utilized in commercial applications. The efficiency of spsp test in addition to not being a complete primality test encourage scientists to open new avenues for  research to improve the test. In that respect, the work \cite{PSW} presented an idea that aim to find all composite numbers which are the elements of spsp to certain bases. A partial answer, which gives all elements of spsp less than a certain integer, is also presented in it. The idea later expanded in many ways \cite{JAE, ZHA1, ZHA2}. For example, suppose one finds the set $S$ of all composite numbers less than $t$ which are spsp to prime bases $p_1,\dots,p_s$ then we have an efficient primality test for all numbers less than $t$. In fact, let $n<t$ be an element of $spsp(p_1,\dots,p_s)$ and $n\not\in S$ then $n$ must be a prime integer. Therefore, the research has been focused on to classify composite strong pseudo prime numbers to bases 2, 3, 5 and so on. For example, if $n$ is in $spsp(2)$, intuitively  $\gcd(n-1,p-1)$ must be large for any prime factor $p$ of $n$. In fact, the experimental results show that for some prime factors $p$ of $n$ we have $\gcd(d(n-1),(p-1))=p-1$ and $d<20$ for $n$ with at most 15 digits \cite{JAE}. \\
 \indent The research on the classification of strong pseudo primes to small bases revealed  useful information. The experimental results in the work \cite{JAE} shows that if $n\in spsp(2)$ and composite then $n$ has only two prime factors, say $p$ and $q$, with a probability more than $95\%$. In addition to this, $\gcd(s_1(p-1),d_1(q-1))=p-1$ and $\gcd(s_2(p-1),d_2(n-1))=p-1$ and $s_1,s_2,d_1,d_2<20$.\\
 \indent In this research we aim to construct a method to detect compositeness of integers which are the members of  $spsp(2)$. In this respect, we first present the properties of strong pseudo primes to base 2. Then a method which detects compositeness of such numbers will be presented. Even though, it is theoretically hard to define all types of $n\in spsp(2)$, the experiment results show that our method is a primality test for all integers with less than 22 digits. In addition to this, the algorithm detects all strong pseudo prime integers in set $S$ which is in \url{http:\\kripto.itu.edu.tr}. In the next section, we introduce our main tool singular cubic and then present the method.

\section{Singular Cubics}   
   
 \indent    Let $\mathbb F_q$ be a finite field with $q$ elements. The integer $q$ is considered to be a prime number from now on. A singular cubic curve $E$ is defined by an equation $$E:y^2=x(x-a)^2$$ for some $0\ne a\in \mathbb F_q$.  Any point on the curve except $(a,0)$ is called a non-singular point. The set of all non-singular points in addition to a point at infinity form an abelian group \cite{WAS}. Let denote this group by $E(\mathbb F_q)$.  The following theorem gives the structure the group $E(\mathbb F_q)$.

\begin{theorem}
	Let $E$ and $\mathbb F_q$ be as above. The group $E(\mathbb  F_q)$ is cyclic and  of order
	\begin{enumerate}
		\item $q-1$ if $a$ is a square in $\mathbb F_q$.
		\item $q+1$ if $a$ is not a square in $\mathbb F_q$.
	\end{enumerate}
\end{theorem} 
  
  \begin{proof}
  	See Theorem 2.31 in \cite{WAS}.
  \end{proof}

It is very easy to determine a point in $E(\mathbb F_q)$. In fact,  for any $b\ne a\in\mathbb F_q$, $(b^2,b(b^2-a))$ is an element of $E(\mathbb F_q)$. The group operation in $E(\mathbb F_q)$ can be performed as performing group operation in an elliptic curve group [Chapter 2, \cite{WAS}]. The curves that are exploited in this work are of special form, that is, there are all defined by equations $$y^2=x(x-a)^2 \text{ over } \mathbb F_q.$$
In other words, these cubic curves are actually nodal curves \cite{NarOzd}. Group operation and representation of elements in respective Jacobian group can be performed in a more simple way. If we apply these to our special curves, we get the following theorem.
\begin{theorem}
	Let $E$ be a singular cubic defined over the field $\mathbb F_q$ defined by an equation $$y^2=x(x-a)^2.$$ Any point $(x,y)$ on the curve $E(\mathbb F_q)$ is uniquely represented by a constant $t$ such that $$t^2\ne a.$$ In addition, suppose that $D_1=t_1,D_2=t_2 \in E(\mathbb F_q)$ then $D_1+D_2$ is defined by $$\dfrac{t_1t_2+a}{t_1+t_2} \text { in } \mathbb F_q.$$ 
\end{theorem}

\begin{proof}
 The proof is based on Algorithm 1 in \cite{NarOzd}. Our singular curve $E$ is defined with the equation $y^2=xf^2(x)$ where $f(x)=x-a$. Any point on the Jacobian of the curve is defined by a single polynomial $h(x)$ such that $\deg h(x) <\deg f(x)$. As $f(x)=x-a$ in our case, $h(x)$ must be a constant. 	Let $D_1$ and $D_2$ are represented by $t_1,t_2\in \mathbb F_q$ respectively. Then $$D=D_1+D_2$$ is represented by $t$ such that
 $$t\equiv g_1(x-a)+g_2(t_1t_2+x) \mod x-a$$ where
 $$g_1(x-a)+g_2(t_1+t_2)=1.$$ 
 
 In this setup, we get $$g_1=0 \text{ and } g_2=\dfrac{1}{t_1+t_2}$$
 Hence $$t\equiv \dfrac{t_1t_2+x}{t_1+t_2} \equiv \dfrac{t_1t_2+a}{t_1+t_2} \mod x-a$$
 \end{proof}

 The pseudo code of the method which is based on above theorem is given below in Algorithm \ref{GrpOp}. Note that Algorithm \ref{GrpOp} is for all integer $n$ and it returns either addition of two group elements or a factor of $n$ if it is a composite integer. Even though we can also use the regular group operation defined for elliptic curves, our algorithm requires only two multiplication and one inversion for an addition in any ring $\mathbb Z_n$.\\

 \begin{algorithm}[H]\label{GrpOp}
 	\SetAlgoLined
 	\KwResult{$Q=P+K$ on Jacobian of $E:y^2=x(x-a)^2$ over $\mathbb Z_n$}
 	$P=t$\;
 	$K=k$\;
 	\KwIn{$t, k, a, n$}

 	\If{$t=0$}{return $k$\;}
 	\If{$k=0$}{return $t$\;}
 	\eIf{$t=0$ and $k=0$ }{return 0\;}{
 		
 		\eIf{$\gcd((t+k),n)!=0$}{
 			\If{$\gcd((t+k),n)=n$}{return 0;}
 			
 			\If{$\gcd((t+k),n)=1$}{ $inv=\dfrac{1}{t+k} \mod n$\;
 				\vspace*{0.2 cm}
 				return $(tk+a)\cdot inv \mod n$\; }
 			
 		}{return $\gcd(t+k,n)$;}
 		
 	}

 	\caption{Group Operation}
 \end{algorithm}

\section{Primality Test} Let $n$ be an element of $spsp(2)$. The algorithm first selects an integer $a$ such that Legendre's symbol $$\left(\dfrac{a}{n}\right)=-1$$ 
The selection can be efficiently performed with quadratic reciprocity law. The quadratic reciprocity law indicates
  
$$\left(\dfrac{a}{n}\right)=(-1)^{\frac{(a-1)(n-1)}{4}} \left(\dfrac{n}{a}\right)$$ 

\indent Therefore, for the sake of computational complexity, the integer $a$ is selected to be a prime number.  At the current state of the art, the only way to find such a prime integer $a$ is the trial\&error method. \\ 
\indent Let $E$ be a curve defined by $y^2=x(x-a)^2$ over $\mathbb Z_n$. The group structure of the curve over $\mathbb Z_n$ is very well known if the integer $n$ is prime. Computing in $E(\mathbb Z_n)$ can be performed by the chord-tangent method even if $n$ is not a prime number \cite{COH}. The next theorem stated by H. Lenstra \cite{LEN} exploits addition algorithm in the elliptic curve groups over the ring $\mathbb Z_n$ for the factorization.
\begin{theorem}\label{Thm1}
	Let $E$ be an elliptic curve over the ring $\mathbb Z_n$ and $p$ be a factor of $n$.  Let $Q$ be an element of $E(\mathbb Z_n)$. $Q$ is also an element of $E$ over $\mathbb Z_p$. Suppose that $Q$ has order dividing $t$ in $E(\mathbb Z_p)$ and the order of $Q\in  E(\mathbb Z_r)$ is coprime to $t$  for any prime factor $r$ of $n$.  Then computing $tQ$ in $E(\mathbb Z_n)$ fails and gives the factor $p$.  

\end{theorem} 

 Lenstra's factorization algorithm in some sense is a modified version of $p-1$ method \cite{Pollard} and  it   is very efficient to find a small factor of composite integers. The proposed method in this work is also considered to be a modified version of combining  $p-1$ and $p+1$ methods \cite{Pollard, Wil}. In addition to this, our method also exploits the structure of the group $E(\mathbb Z_n)$.   \\ 
 \indent Let $n$ be an odd integer. Let us define a  singular cubic $E$  with the equation  $$E:y^2=x(x-a)^2  \text{ on }\mathbb Z_n.$$  
 
 If the integer $n$ is prime, then the order of $E(\mathbb Z_n)$ must be $n+1$. This fact is also being employed by the algorithm to catch composite numbers. The steps of our algorithm are presented below.
  \subsection{Primality Test Algorithm}\label{Alg1}
  \begin{enumerate}
  	\item Check if $n$ is  strong pseudo prime to base 2. If it is not then declare  $ n $  composite, otherwise go to the next step.
  	\item Start $a=5$ and check whether $\left(\dfrac{n}{a}\right)=-1$. If yes, proceed to next step. Otherwise, change $a$ to the next prime such that $a\equiv 1\mod 4$ until $\left(\dfrac{n}{a}\right)=-1$.
  	\item Define the curve $E$ with equation  $y^2=x(x-a)^2$ on $\mathbb Z_n$.
  	\item Find a point $P \mod n$ whose order is not 100 or less and $2bP$ is not identity $\mod n$ where $b$ is the multiplication of all primes less than 100. Note that for any $\ell\in\mathbb Z_n$, $(\ell^2,\ell(\ell-a))$ is a point on $E \mod n$.
  	\item Compute $Q=b(n-1)P+2bP \mod n$ where $b$ is again the multiplication of all primes less than 100. Note that $Q=b(n+1)P$.
  	\item If $Q$ is not identity in $E(\mathbb Z_n)$, then $n$ is composite otherwise $n$ is prime. 
  \end{enumerate} 
  
  The below discussion describes evidences that the above algorithm might be  a primality test.
 
 The algorithm starts with an integer $n$ which passes the strong pseudo prime test to base 2. Several studies have been conducted to give an explicit characterization of such composite integers \cite{JAE, PSW, ZHA1,ZHA2}. The experimental results show that a majority of such integers $n$ have only two factors. For example, the work \cite{JAE} indicates that more than $95\%$ of the $n$ have two factors $p$ and $q$ such that $$s(q-1)=d_1(p-1) \text{ and } p-1\mid k(n-1) \text{ and } q-1\mid d_2(n-1)  \text{  with } s,d_1,d_2,k \le 100.$$
  
 \begin{theorem}
 Let $n$ be a strong pseudo prime to base 2 and the primes $p, q$ are the only factors of $n$. If $p$ and $q$ are of above form, then the algorithm catches compositeness of such an integer $n$. 
 \end{theorem} 

\begin{proof}
The prime integer $a$ is selected to be $$\left(\dfrac{a}{n}\right)=-1.$$ This implies $a$ is a quadratic non-residue modulo $p$ or $q$ but not both.\\
\begin{description}
	\item [Case 1] $$ \left(\dfrac{a}{p}\right)=-1 \text{ and } \left(\dfrac{a}{q}\right)=1$$
	Note that $$q-1\mid t_1(n-1) \text{ for some integers } t_1 \le 100.$$ 
	The point $P$ in $E(\mathbb Z_q)$ has order dividing $q-1$ and in $E(\mathbb Z_p)$ has order dividing $p+1$. Let  $b$ be the multiplication of all primes less than 100. Then $$b(n-1)P =\infty \mod q$$
	
	If also $b(n-1)P =\infty \mod p,$ then $$b(n-1)P=\infty \mod n.$$ This means that $$Q=b(n-1)P+2bP=b(n+1)P\ne\infty \mod n$$  That says $n$ is not a prime number.\\
	On the contrary, suppose that  $b(n-1)P\ne \infty \mod p$.  In this case, computing $b(n-1)P\mod n$ fails and returns $q$. The argument shows that in Case 1, the algorithm definitely catches the compositeness of $n$.\\
	
\item [Case 2]	

  $$ \left(\dfrac{a}{p}\right)=1 \text{ and } \left(\dfrac{a}{q}\right)=-1$$
\noindent As $t_3(n-1)=t_4(p-1)$ for some integers $t_3,t_4\le 100$. A similar argument in Case 1 shows the algorithm catches compositeness of the $n$.

\end{description}
\end{proof} 
   
 	\begin{theorem}
 	Let $n$ be a strong pseudo prime to base 2 and the primes $p, q \text{ and } r$ are the only factors of $n$. If the values of $ \gcd(n-1,p-1), \gcd(n-1,q-1) \text{ and } \gcd(n-1,r-1) $ are large, then the algorithm might catch compositeness of such an integer $n$. 
 	\end{theorem}
 \begin{proof}
 The prime integer $a$ is selected to be $$\left(\dfrac{a}{n}\right)=-1.$$ If $a$ is a quadratic non-residue for only one factor of $ n $, then the above theorem can be applied. Now we handle the case that $a$ is a quadratic non-residue modulo $p, q$ and $ r $, that is,   $$ \left(\dfrac{a}{p}\right)=-1 \text{ and } \left(\dfrac{a}{q}\right)=-1 \text{ and } \left(\dfrac{a}{r}\right)=-1.$$ Since $ n $ is a strong pseudo prime to base 2, $$ \#2\mid p-1,  \#2\mid q-1 \text{ and } \#2\mid r-1 $$ where  \#2  is the order of 2 in the respected modular group. Basically, it says $\gcd(n-1,p-1)$, $\gcd(n-1,q-1)$ and $\gcd(n-1,r-1)$ are large. On the other hand, for any point $ P $ in $E(\mathbb Z_n)$ with the order dividing $ n+1 $, it is obvious that the order of $P$ in $E(\mathbb Z_p)$ divides $p+1$, in $E(\mathbb Z_q)$ divides $q+1$ and in $E(\mathbb Z_r)$ divides $r+1$. Therefore, if $(n+1)P$ is the identity $\mod n$ then $\gcd(n+1,p+1)$, $\gcd(n+1,q+1)$ and $\gcd(n+1,r+1)$ are all large. However, it is highly improbable that not only the numbers $p-1, q-1, r-1$ but also  $p+1, q+1, r+1$ have large common divisors. 
\end{proof}

Note that intuitively $n\in spsp(2)$ implies $\gcd(n-1,p-1)$, $\gcd(n-1,q-1)$ and $\gcd(n-1,r-1)$ are large when $n=pqr$. Step 4 of Algorithm \ref{Alg1} keeps us staying on the safe side. On the other hand, we believe it is not necessary to check a random point $P$ on $E(\mathbb Z_n)$ has order less than 100. Algorithm \ref{PTest} offers a less costly version of Algorithm \ref{Alg1}.   We conduct heuristics tests with Algorithm \ref{PTest} to see if it misses any composite number, but we were not able to find one. Basically, the test checks all integers less than $10^{21}$ and it catches all composite ones. The further experiments involving integers from a couple hundreds to several thousands digits were also able to discover compositeness of  integers. Therefore, we believe  that Algorithm \ref{PTest} will never fail to catch compositeness any integer and we conclude the work with the following conjecture.\\

\begin{algorithm}[H]\label{PTest}
	\SetAlgoLined
	\KwResult{$n$ is prime or not. }
	\KwIn{$n$\;}
	\eIf {$n\in spsp(2)$ }{
		$a=5$\;
		$i=0$\;
		\While{$i=0$}{	
			\eIf  {$a\equiv 1 \mod 4$ and $\left(\dfrac{n}{a}\right)=-1$}{$i=1$\;}{$a=$nextprime($a$)\;}
		}
		
		$E: y^2=x(x-a)^2$\;
		$P= t$\;
		
		$Q=(n-1)P+2P$\;
		\eIf{$Q=\text{identity}$}{
			return true\;
		}{
			return false\;
		}
	}
	{
		return false\;
	}
	\caption{Primality Test}
\end{algorithm}
\hspace*{1 cm}
\\

\noindent{\bf Conjecture:} Algorithm \ref{PTest} is a primality test with a running time of $O(\log^2 n)$ for a given integer $n$. \\
\indent In order to illustrate benefit of proving or accepting above conjecture, we present timings (miliseconds) of  primality test algorithms for integers of different size. \\
\\

\begin{tabular}{|c|c|c|c|}
\hline
\text{{\bf Size (bits)} } & \text{{\bf ECPP}}&\text{ {\bf Cyclotomic Field Test}}&{\text {\bf Singular Cubic Test}}\\
\hline
256 & 28.5 & 51 &2.4\\
\hline
512 & 398.8&497 &9.4\\
\hline
1024 & 4691.5 & 14880 &42.1\\
\hline
2048 & 87375.6 & 413408 &210.6\\
\hline

\end{tabular}

\vspace*{0.5 cm}
 \indent The test are conducted on a computer with operating system  macOS, 16 GB memory, Intel i5 3.8 GHz CPU. The codes are written in C++ language.

\bibliographystyle{amsplain}

\end{document}